\documentclass[11pt]{amsart}

\usepackage{scrextend}

\newcommand{\set}[1]{\left\lbrace #1 \right\rbrace}

\makeatletter
\newcommand\footnoteref[1]{\protected@xdef\@thefnmark{\ref{#1}}\@footnotemark}
\makeatother

\usepackage{amsmath,amsfonts,amssymb,amsthm,url,verbatim}
\usepackage[dvips]{graphicx}
\usepackage[croatian,english]{babel}
\usepackage{amsfonts}
\usepackage[utf8]{inputenc}
\usepackage[T1]{fontenc}
\usepackage{amsmath}
\usepackage{amssymb}
\usepackage{latexsym}
\usepackage[usenames,dvipsnames]{xcolor}
\usepackage{soul}
\usepackage[colorinlistoftodos,prependcaption]{todonotes}
\usepackage[all]{xy}
\usepackage{booktabs}
\usepackage{longtable}
\usepackage{multirow}

\usepackage[backref]{hyperref}

\usepackage{ifthen}
\usepackage{graphpap}
\usepackage{xargs}
\usepackage{cleveref}

\newcommand{\R}{\mathbb{R}}
\newcommand{\Q}{\mathbb{Q}}
\newcommand{\C}{\mathbb{C}}
\newcommand{\Z}{\mathbb{Z}}

\newcommand{\F}{\mathbb{F}}

\newcommand{\PP}{\mathbb{P}}

\DeclareMathOperator{\Pic}{Pic}
\DeclareMathOperator{\Div}{div}

\DeclareMathOperator{\Aut}{Aut}
\DeclareMathOperator{\spec}{Spec}

\newtheorem{theorem}{Theorem}[section]
\newtheorem{proposition}[theorem]{Proposition}
\newtheorem{lemma}[theorem]{Lemma}
\newtheorem{corollary}[theorem]{Corollary}

\theoremstyle{definition}
\newtheorem{definition}[theorem]{Definition}
\theoremstyle{remark}

\newtheorem{remark}[theorem]{Remark}
\newtheorem{example}[theorem]{Example}

\DeclareMathOperator{\Gal}{Gal}
\DeclareMathOperator{\Sym}{Sym}

\DeclareMathOperator{\Hom}{Hom}

\def\diam#1{\langle#1\rangle}
\newcommand{\Cn}[1]{\Z/{#1}\Z}

\title{Torsion of elliptic curves over cyclic cubic fields}
\author{Maarten Derickx}
\address{Johann Bernoulli Institute, Universiteit Groningen, Nijenborgh 9,
	9747 AG Groningen, The Netherlands}
\email{maarten@mderickx.nl}
\author{Filip Najman}
\address{University of Zagreb, Bijeni\v{c}ka cesta 30, 10000 Zagreb, Croatia}
\email{fnajman@math.hr}

\date{\today}
\makeatletter
\providecommand\@dotsep{5}
\renewcommand{\listoftodos}[1][\@todonotes@todolistname]{%
	\@starttoc{tdo}{#1}}
\makeatother
\newcommandx{\filip}[2][1=]{\todo[linecolor=blue,backgroundcolor=blue!25,bordercolor=blue,#1]{#2}}
\newcommandx{\maarten}[2][1=]{\todo[linecolor=green,backgroundcolor=green!25,bordercolor=green,#1]{#2}}

\begin{document}
\begin{abstract}{We determine all the possible torsion groups of elliptic curves over cyclic cubic fields, over non-cyclic totally real cubic fields and over complex cubic fields.}
\end{abstract}
\thanks{The second author was supported by the QuantiXLie Centre of Excellence, a project
co-financed by the Croatian Government and European Union through the
European Regional Development Fund - the Competitiveness and Cohesion
Operational Programme (Grant KK.01.1.1.01.0004).}

\maketitle

\section{Introduction}
An important problem in the theory of elliptic curves is to determine the possible torsion groups of elliptic curves over number fields of degree $d$. Mazur solved the problem for $d=1$ \cite{maz} and Kamienny \cite{kam}, building on the work of Kenku and Momose \cite{km}, solved the problem for $d=2$. Recently, Derickx, Etropolski, van Hoeij, Morrow and Zuerick-Brown announced the solution of the problem for $d=3$ \cite{dehmz}, building on the work of Parent \cite{par1,par2}. For $d>3$ the problem remains unsolved at the moment.

A question that naturally arises is which of the groups that arise as torsion groups of elliptic curves over number fields of degree $d$ arise over some natural subset of the set of number fields of degree $d$. These subsets of the set of all number fields of degree $d$ can be chosen to be, perhaps most naturally, the subset of real (or totally real) number fields, the subset of complex number fields, and the subset of number fields whose normal closure over $\Q$ has Galois group isomorphic to some prescribed group $G$. Throughout the paper, by abuse of language we say that a number fields $K$ has Galois group $G$ over $\Q$ if the Galois group over $\Q$ of the normal closure of $K$ over $\Q$ is $G$.

These problems are of course meaningless for the case $d=1$, and for $d=2$ one can only consider the subdivision of quadratic fields into real and imaginary quadratic fields. The possibilities in each of these cases (for $d=2$) follow directly from \cite{bbdn} and are summed up in \Cref{d2_thm}.
Thus the problem has been solved completely for $d=2$, so it is natural to consider the case $d=3$, which is the last case where all the possible torsion groups are known. The main result of the paper is the determination of all possible torsion groups of elliptic curves over
\begin{itemize}
\item[a)] All cubic fields with Galois group $\Z/3\Z$.
\item[b)] All complex cubic fields.
\item[c)] All totally real cubic fields with Galois group $S_3$.
\end{itemize}

Parts of some proofs of our results are based on computations in Magma \cite{magma}. Some of these computations depend on a magma package writen by Solomon Vishkautsan and the first author\cite{dv}.  All computations done in Magma for this paper can be found at \url{https://github.com/wishcow79/chabauty/tree/master/papers/cyclic_cubic}.

We would like to thank Solomon Vishkautsan for making his Chabauty code public and his help in explaining on how to use it.

\section{Strategy for determining the Galois groups of degree $d$ points on a curve}
	
	The problem described in the introduction easily translates into the question which cubic fields can occur as the field of definition of degree $3$ points on certain modular curves. Ignoring the totally real and complex cubic field questions for a moment and only focussing on the Galois group, one is then led to the more general question which Galois groups can occur as the Galois group of a degree $d$ point on some curve $X/\Q$. For the rest of this section $X/\Q$ will be a smooth projective and geometrically irreducible curve.
	
	Let $L/\Q$ be a degree $d$ extension, $\widetilde L$ be its normal closure and $G := \Gal(\widetilde L /\Q)$ its Galois group. A choice of an enumeration $\sigma_1,\ldots,\sigma_d$ of the $d$ elements in $\Hom_\Q(L,\widetilde L)$ allows one to see $G$ which naturally acts on $\Hom_\Q(L,\widetilde L)$ as acting on $1,\ldots,d$ so that $G \subseteq S_d$. Let now $H \subset S_d$ be any subgroup, then $H$ naturally acts on $X^d$ by permuting the factors of the direct product, the corresponding quotient $X^d/H$ will be called the $H$-symmetric product of $X$ and denoted by $X^{(H)} := X^d/H$. If $H=S_d$ then the $H$ symmetric product $X^{(S_d)}$ is nothing other then the $d$-th symmetric product $X^{(d)}$, note that contrary to $X^{(d)}$ it is not always true that $X^{(H)}$ is smooth, the simplest non-smooth example is $X^{(\Z/3\Z)}$, however $X^{(H)}$ is smooth if one stays away from the fixed points of $H$, in particular it is smooth away from the image of $\Delta(X^d)$, the locus in $X^d$ where at least two coordinates coincide. If one lets $s \in X(L)$ be a point and define $s^{(\sigma)} := (\sigma_1(s),\ldots,\sigma_d(s)) \in X^d(\widetilde L)$ and let $s^{(H)}$ be its image in $X^{(H)}(\widetilde L)$, then by definition one has that $s^{(H)} \in X^{(H)}(\Q)$ if one has that the action of $G = \Gal(\widetilde L /\Q)$ on $s^{(H)}$ is trivial. By construction the latter happens if $G \subseteq H$ as subgroups of $S_d$, and if additionally $s \in X(L)$ is not definable over any subfield of $L$ then one even has that the action of $G$ on $s^{(H)}$ is trivial if and only if $G \subseteq H$.
	
	\begin{remark}In fact there is a converse to this construction, namely every $x \in X^{(H)}(\Q)$ that does not come from some $X^{(H')}(\Q)$ for some $H'\subsetneq H$ has to be of the from $s^{(H)}$ as above.
	\end{remark}
So this turns the study of degree $d$ points on $X$ whose Galois group is isomorphic to $H$ to a study of the rational points of $X^{(H)}$ (that do not come from $X^{(H')}(\Q)$ for some $H' \subsetneq H$). This can be a difficult problem in general depending on what $X$, $d$ and $H$ are. For example for $X=\PP^1$ this problem is equivalent to the inverse Galois problem for $H$, and for $X=E$ an elliptic curve and $H = \Z/d\Z$ this is closely related to the rank growth of $E$ over cyclic extensions - see for example \cite[Section 4]{fhk}.

The cases that need to be dealt with in this paper are however more feasible than the general problem, because we restrict only to $d=3$. In this case there are only the groups $S_3$ and $\Z/3\Z$ to consider and the question becomes which points on $X^{(3)}(\Q)$ come from $X^{(\Z/3\Z)}(\Q)$ and which do not. For most of the relevant modular curves that actually have a cubic point it is easy to construct infinitely many points on $X^{(3)}(\Q)$ that come from $X^{(\Z/3\Z)}(\Q)$ as well as infinitely many that don't (see \Cref{mainthm1,mainthm2,mainthm3} and their proofs, as well as \cite{jeon}). This is done for example by constructing functions $f \in \Q(X)$ such that the function field extension $\Q(f) \subset \Q(X)$ has either Galois group $\Z/3\Z$ or $S_3$. The most problematic cases are $X_1(16)$ and $X_1(20)$, both of which have functions of degree $3$ where the function field extension has Galois group $S_3$, but there are no obvious constructions for cubic points with Galois group $\Z/3\Z$. However, in these two cases one can prove that every point of degree $3$ on $X_1(N)$ comes from a function $f : X_1(N) \to \PP^1$ of degree $3$ (see \Cref{cor:inverse_image}) and that additionally there are only finitely many functions of degree $3$ (see \Cref{lem:16_1,lem:20_1}). So the question is now what the Galois group of $f^{-1}(t)$ can be for $t \in \PP^1(\Q)$.

Going back to the more general case, let $f : X \to \PP^1$  be a function of degree $d$  and $\widetilde{\Q(X)}$ be the normal closure of $\Q(X)$ seen as a finite extension of $\Q(\PP^1)$ and let $H \subset S_d$ be its Galois group. The generic point of $X$ gives a point $\eta \in X(\Q(X))$.  Since $\Gal(\widetilde{\Q(X)}/\Q(\PP^1)) = H$ one has that $\eta^{(H)} \in X^{(H)}(\widetilde{\Q(X)})$ is actually a $\Q(\PP^1)$-valued point. Define $f^* : \PP^1 \to X^{(H)}$ to be the morphism corresponding to $\eta^{(H)}$. Now if $t \in \PP^1(\Q)$ is not a ramification point and $s \in f^{-1}(t)$ is a point whose field of definition $L$ is of degree $d$, then one has $f^*(t) = s^{(H)}$. So the question whether the Galois group of $L$ is smaller than $H$ becomes equivalent to whether $f^{*}(t) \in X^{(H)}(\Q)$ comes from a rational point on $X^{(H')}$ for some $H' \subsetneq H$. Now for a $H' \subsetneq H$ define $C_{H'} := X^{(H')} \times_{X{(H)}} \PP^1$
$$\xymatrix{C_{H'} \ar[r]\ar[d] & \PP^1 \ar[d]^{f^*} & \spec \widetilde{\Q(X)}^{H'} \ar[d]^{\eta^{(H')}}\ar[r] & \spec \Q(\PP^1) = \spec \widetilde{\Q(X)}^{H} \ar[d]^{ \eta^{(H)}}\\
X^{(H')} \ar[r] &X^{(H)} & X^{(H')} \ar[r] &X^{(H)}}
$$
Now $C_{H'}$ is a curve, so the question whether the Galois group can get smaller than $H$ has been turned into a question about rational points on curves. The generic fiber of $C_{H'}$ is just $\widetilde{\Q(X)}^{H'}$ with the map $\eta^{(H')}$. This means that it is possible to compute explicit equations for the normalisation of $C_{H'}$. In the case that $H=S_d$ and $H'=A_d$ one can even be more explicit. Indeed if one writes $\Q(X) \cong \Q(f)[t]/g$ for some $g$ of degree $d$ in $t$ then $\Q(C_{A_d}) \cong \Q(f)[t]/(t^2-\Delta(g))$.

\section{Previously known results}
In this section we describe known results, some of which we will use in our proofs.

As mentioned in the introduction, for quadratic fields we have a classification of which torsion groups appear over real and which over imaginary quadratic fields. Although the result is not used in this paper, we state it for completeness, as it is not explicitly stated in \cite{bbdn}, although it follows directly from the results proved in this paper.

\begin{theorem}\label{d2_thm} \cite{bbdn}
\begin{itemize}

\item[a)] Let $E$ be an elliptic curve over an imaginary quadratic number field $K$. Then $E(K)_{tors}$ is isomorphic to one of the following groups:
\begin{align*} \Z/n\Z & \text{ for } n=1,\ldots ,12, 14,15,16\\
\Z/2\Z \oplus\Z/2n\Z  & \text{ for } n=1,\ldots ,6\\
\Z/3\Z \oplus \Z/3n\Z& \text{ for } n=1,2\\
\Z/4\Z \oplus \Z/4\Z. & \
\end{align*}
Each of the cases arises for infinitely many non-isomorphic elliptic curves.
\item[b)] Let $E$ be an elliptic curve over a real quadratic number field $K$. Then $E(K)_{tors}$ is isomorphic to one of the following groups:
\begin{align*} \Z/n\Z & \text{ for } n=1,\ldots ,16,18\\
\Z/2\Z \oplus\Z/2n\Z  & \text{ for } n=1,\ldots ,6\\
\end{align*}
Each of the cases arises for infinitely many non-isomorphic elliptic curves.
\end{itemize}
\end{theorem}

The following result tells us which are the possible torsion groups of elliptic curves over all cubic fields.

\begin{theorem}\cite{dehmz} \label{thm_dehmz}
The possible torsion groups of elliptic curves over cubic fields are
$$\Cn n \text{ for }n= 1\ldots 16, 18, 20,21,$$
$$\Cn 2 \times \Cn {2m} \text{ for }m=1\ldots 7.$$
There exists $1$ curve with torsion $\Cn {21}$.
\end{theorem}

\begin{remark}\label{rem1}
The one curve with $\Cn {21}$ torsion was found in \cite{naj} and is defined over a cyclic cubic field.
\end{remark}

Jeon \cite{jeon} determined the groups that appear over cyclic cubic fields infinitely often.
\begin{theorem}\cite{jeon} \label{thm:jeon}
There exists infinitely many elliptic curves over cyclic cubic fields with torsion:
$$\Cn n \text{ for }n= 1\ldots 15, 18,$$
$$\Cn 2 \times \Cn {2m} \text{ for }m=1\ldots 7.$$
No other torsion group appears over infinitely many cyclic cubic fields.
\end{theorem}
In this paper, we will among other things solve the problem of which torsion groups appear at all (not just infinitely often) over cyclic cubic fields.

A recent result of Bruin and the second author shows that the torsion group $\Cn 2 \times \Cn {14}$ appears only over cyclic cubic fields.

\begin{theorem}\cite[Theorem 1.2.]{bn}
If an elliptic curve $E$ over a cubic field $K$ has torsion $\Cn 2 \times \Cn {14}$, then $K$ is cyclic.
\end{theorem}

\section{Torsion groups over cyclic cubic fields}
In this section we classify the possible torsion groups of elliptic curves over cyclic cubic fields.

We want to prove:
\begin{theorem}\label{mainthm1}
Let $E$ be an elliptic curve over a cyclic cubic field $K$. Then $E(K)_{tors}$ is isomorphic to one of the following groups:
$$\Cn n \text{ for }n= 1\ldots 16, 18, 21,$$
$$\Cn 2 \times \Cn {2m} \text{ for }m=1\ldots 7.$$
For each group $G$ from the list above there exists a cyclic cubic field $K$ and an elliptic curve $E/K$ such that $E(K)_{tors}\simeq G$.
\end{theorem}

From \Cref{thm:jeon}, it follows that all of the groups in the list apart from $\Cn {16}$ and $\Cn {21}$ appear. The group $\Cn {21}$ appears by \Cref{rem1}. Thus, by \Cref{thm_dehmz}, it remains to show that $\Cn {16}$ appears (finitely many times) and $\Cn {20}$ does not. To do this we have to show that $X_1(20)$ has no non-cuspidal cubic points over any cyclic cubic field and find all such points on $X_1(16)$.

\begin{definition}
	Let $C$ be a curve over a perfect field $K$ then the {\it new part} of $C^{(d)}(K)$ is defined to be
	$$C^{(d)}(K)^{new} := C^{(d)}(K) \setminus \bigcup_{i=1}^{d-1}\iota_i(C^{(i)}(K) \times C^{(d-i)}(K)).$$ Where $\iota_i : C^{(i)} \times C^{(d-i)} \to C^{(d)}$ is the map sending effective divisors of degree $i$ and $d-i$ to their sum.
\end{definition}
\begin{remark}If $K$ is algebraically closed then $C^{(d)}(K)^{new} = \emptyset$ for all $d>1$ because each $\iota_i$ is finite surjective and hence surjective on $K = \overline K$ points.
\end{remark}
\begin{remark}
	If one thinks about $C^{(d)}(K)$ as effective divisors of degree $D$ on $C(\overline K)$ that are stable under the action of Galois. Then $C^{(d)}(K)^{new}$ are exactly those divisors $D = \sum_{i=1}^n m_iP_i$ of degree $d$ where all the multiplicities $m_i$ of the points are $1$ and the action of $\Gal(K)$ is transitive on the points $P_1, \ldots, P_n$. In particular  $C^{(d)}(K)^{new}$ consists exactly of those points in $C^{(d)}(K)$ of the form $x^{(d)}$ where $x \in C(\overline K)$ is a point whose field of definition is of degree $d$ over $K$.
\end{remark}
\begin{lemma}\label{lem:surjection_implies_higher_surjections} Let $C$ be a curve over a field $K$ and $Z \subset C$ a closed subscheme with $Z(K) \neq \emptyset$ and suppose that $d$ is an integer such that $\mu_d : Z^{(d)}(K) \to \Pic^d(C)(K)$ given by $D \mapsto \mathcal O_C(D)$ is surjective, then $\mu_e : Z^{(e)}(K) \to \Pic^e(C)(K)$ is surjective for all integers $e>d$ as well.
\end{lemma}

\begin{proof}
	Let $x \in Z(K)$ be a point and define $\phi_x : \Pic^d(C) \to \Pic^e(C)$ be the map given by $\phi_x(\mathcal L) = \mathcal L((e-d)x)$. Then $\phi_x$ is an isomorphism and hence any $\mathcal \in \Pic^e(C)(K)$ can be written as $\mathcal L'((e-d)x)$ with $\mathcal L' \in \Pic^d(C)(K)$, by the surjectiveness assumption on $\mu_d$ we can even take $\mathcal L' = \mathcal O_C(D)$ with $D\in Z^{(d)}(K)$. This gives  $$L = L'((e-d)x) = O_C((e-d)x+D)= \mu_e((e-d)x+D)$$  where $(e-d)x+D$ is in the domain of $\mu_e$ because $e - d>0$. It follows that $\mu_e$ is surjective.
\end{proof}

\begin{lemma}\label{lem:surject_implies_function}
	Let $C$ be a curve over a field $K$ with $C(K) \neq \emptyset$, $d$ is an integer and $S \subset C^{(d)}(K)$ a subset such that $\mu_{\mid S} : S \to \Pic^d(C)(K)$ given by $D \mapsto \mathcal O_C(D)$ is surjective, then every point in $C^{(d)}(K)^{new} \setminus S$ is of the form $f^*(x)$ where $f \in K(C)$ is a function of degree $d$, $x \in \PP^1(K)$ and $f^*(x)$ denotes the pullback of $x$ along $f$ as divisor.
\end{lemma}
\begin{proof}
	Let $D \in C^{(d)}(K)^{new} \setminus S$ then by the surjectiveness assumption there is an $D' \in S$ such that $\mu(D') = \mu(D)$. Let $f$ be a function such that $\Div f = D-D'$, then by definition one has  $D \neq D'$, and because $D \in C^{(d)}(K)^{new}$ one even has that the supports of $D$ and $D'$ have to be disjoint. In particular $f$ has to have degree exactly $d$, and $D = f^*(0)$ by construction.
\end{proof}

\begin{proposition}\label{prop:cusps_surject}
	Let $C_1(N) := X_1(N) \setminus Y_1(N)$ be the closed subscheme consisting of the cusps. Then the maps $\mu_3: C_1(16)^{(3)}(\Q) \to \Pic^3 X_1(16)(\Q)$ and $\mu_3: C_1(20)^{(3)}(\Q) \to \Pic^3 X_1(20)(\Q)$ are surjective.
\end{proposition}
\begin{proof}
	For $X_1(16)$ this follows from \cref{lem:surjection_implies_higher_surjections} because according to \cite[Lemma 4.8.]{bbdn} the conditions of  \cref{lem:surjection_implies_higher_surjections} are satisfied for $d=2$.
	
	For $X_1(20)$ this is slightly more work. To determine $J_1(20)(\Q)$, an easy computation in Magma shows that $J_1(20)(\F_3)\simeq \Z/ 60 \Z$, and the subgroup of $J_1(20)(\Q)$ generated by $\Q$-rational divisors supported on the cusps is isomorphic to $\Z/60\Z$, showing that $J_1(20)(\Q)$ and hence $\Pic^3 X_1(20)(\Q)$ has cardinality 60.

So it suffices to compute the cardinality of the image of $\mu_3 :C_1(20)^{(3)}(\Q) \to \Pic^3 X_1(20)(\Q)$. The curve $X_1(20)$ has 6 rational cusps, 3 pairs of Galois conjugate cusps over quadratic extensions and no cusps with a cubic field of definition, hence the cardinality of $C_1(20)^{(3)}(\Q)$ is $74 = 56+18$ with a contribution of $\binom{6+3-1}{3}=56$ coming from triples of rational cusps and a contribution of $6 \cdot 3=18$ coming from a rational cusp together with a pair of Galois conjugate cusps. Computing the map $\mu_3 :C_1(20)^{(3)}(\Q) \to\Pic^3 X_1(20)(\Q)$ can be done using modular symbols as in \cite[Section 4]{dkm}, this shows that $\mu_3$ has 54 fibers containing exactly one element, 2 fibers with 4 elements (these two fibers are swapped by the action of the diamond operators) and 4 fibers containing 3 elements (these 4 fibers are again permuted by the diamond operators), this gives a total of 60 nonempty fibers and hence $\mu_3$ is surjective. 
\end{proof}

Combining \cref{lem:surject_implies_function,prop:cusps_surject} one gets the following corollary:
\begin{corollary}
\label{cor:inverse_image}
	Let $N=16$ or $20$, then any point of degree $3$ over $\Q$ on $Y_1(N)$ occurs in an inverse image of the form $f^{-1}(t)$, with $f \in \Q(X_1(N))$ a function of degree $3$ and $t \in \PP^1(\Q)$.
\end{corollary}

The above corollary says that in order to explicitly describe all degree 3 points on either $X_1(N)$ for $N=16$ or 20 one just needs to find all $\Q$-rational functions of degree $3$ on these curves. However  \Cref{prop:cusps_surject} also gives a hint on how to find all these functions explicitly, indeed it says that every degree $3$ function is (up to an automorphism of $\PP^1$) a pole supported at the cusps. This means that \Cref{lem:16_1,lem:20_1} can now easily be proved by computing which elements $C \in C_1(N)^{(3)}(\Q)$ are the pole divisor of a function of degree 3. And then grouping these elements together by linear equivalence and taking one of them under the action of the diamond operators. What follows is a description of what happens for $X_1(16)$ and $X_1(20)$ when doing this computation.

\begin{lemma}
\label{lem:16_1}
There are $4$ (up to the action of the diamond operators on $X_1(16)$ and automorphism of $\PP^1$) maps $X_1(16)\rightarrow \PP^1$ of degree $3$.
\end{lemma}
\begin{proof}

The modular curve $X_1(16)$ has an affine model (see \cite{sut})
$$X_1(16):y^2 =x(x^2+1)(x^2+2x-1).$$ We can identify $C_1(16)(\C)$ with $\Gamma_1(16) \backslash \PP^1(\Q)$ an explicit set of representatives of this is given by $$\{0, 1/8, 1/7, 1/6, 1/5, 1/4, 1/3,1/2, 3/4, 3/8, 3/16, 5/16,  7/16,  \infty\}.$$ Under the action of Galois this splits up into $9$ orbits, namely:
$$ \{0, 1/7, 1/5, 1/3\},  \{1/6, 1/2\},\{1/4, 3/4\}, \{1/8\}, \{3/8\},\{3/16\},  \{5/16\}, \{7/16\}, \{\infty\}.$$
This means that $C_1(16)^{(3)}(\Q)$ consists of $\binom{6+3-1}{3} = 56$ divisors supported at the $6$ rational cusps and $6 \cdot 2$ divisors of the form $P_1 + P_2$ where $P_1$ is one of the $6$ rational cusps and $P_2$ is the sum of the points in either of the two pairs of Galois conjugate cusps. By explicit computation one checks that each of the 12 effective divisors of degree 3 not supported at the rational cusps is linearly equivalent to a divisor supported on the rational cusps, so that up to automorphisms of $\PP^1$ every function of degree $3$ has a pole divisor supported on the rational cusps. It turns out that of the 56 effective divisors of degree 3 supported on the rational cusps, there are exactly 32 that are the pole of a degree $3$ function. These 32 divisors are divided into 10 linear equivalence classes of size 2 and 4 linear equivalence classes of size 3, for a total of 14 linear equivalence classes, hence up to the automorphisms of $\PP^1$ there are exactly 10+4 = 14 functions of degree $3$ on $X_1(16)$. The 10 linear equivalence classes of size $2$ form one orbit of size 2 and two orbits of size 4 under the diamond operators, while the 4 linear equivalence classes of size $3$ are one diamond orbit. So that the $14$ functions form $4$ distinct orbits under the action of the diamond operators. Finally the explicit equations for these functions where obtained by finding an $f$ whose divisor is the difference between divisors in the same linear equivalence class. 
Given a degree 3 function $g_i$,
$$g_i\in\Q(x)[y]/(y^2 -x(x^2+1)(x^2+2x-1))\simeq \Q(X_1(16))$$
one can compute its minimal polynomial over $\Q(x)$ as an element $f_i\in \Q(x)[t]$. If the degree of $f_i$ in $t$ is $[\Q(X_1(16)):\Q(x)]=2$ then $t \mapsto g_i$ gives an isomorphism $\Q(x)[t]/f_i \to \Q(X_1(16))$.

The explicit equations for the 4 maps $X_1(16)\rightarrow \PP^1$ of degree $3$ that we obtain are:
\begin{align}
\label{mp:16_start}
g_1(x,y)&=(x^3 + x^2 + 3x - 1+2y)/(2x^3 - 2x^2 - 2x
+ 2),\\
\label{maps:16_1}
f_1(x,t)&=(4x^3 - 4x^2 - 4x + 4)y^2 + (-4x^3 - 4x^2 - 12x + 4)y + x^3 - x^2 - x + 1,\\
\label{mp:16_2}
g_2(x,y)&=\frac{y+2x^2}{x^2 - 1},\\
\label{maps:16_2}
f_2(x,t)&=(x^2 - 1)t^2 - 4x^2t - x^3 + 2x^2 - x,\\
\label{mp:16_3}
g_3(x,y)&= \frac{y-2x^2}{x^3 - x},\\
\label{maps:16_3}
f_3(x,t)&= (x^3 - x)t^2 + 4x^2t - x^2 + 2x - 1,\\
\label{mp:16_end}
g_4(x,y)&=\frac{2x^2+y}{x^2 - x},\\
\label{maps:16_4}
f_4(x,t)&=(x^2 - x)t^2 - 4x^2t - x^3 + x^2 + x - 1.
\end{align}
\end{proof}

\begin{remark}
Notice that from the computations in the proof of \Cref{lem:16_1} one can also see that the image of $\mu_3: C_1(16)^{(3)}(\Q) \to \Pic^3 X_1(16)(\Q)$ has 20 elements and hence is surjective.
\end{remark}

\begin{lemma}
\label{lem:20_1}
There are $2$ (up to the action of the diamond operators on $X_1(20)$ and automorphism of $\PP^1$) maps $X_1(20)\rightarrow \PP^1$ of degree $3$.
\end{lemma}
\begin{proof}
The modular curve $X_1(20)$ has an affine model (see \cite{sut})
$$X_1(20):y^3= x^3y^2 + x^2y - x.$$

We can identify $C_1(20)(\C)$ with $\Gamma_1(20) \backslash \PP^1(\Q)$ an explicit set of representatives of this is given by $$\{0, 1/10, 1/9, 1/8, 1/7, 1/6, 1/5, 1/4, 1/3, 1/2, 2/5, 3/4, 3/5, 3/8, 3/10, 3/20, 4/5,  7/20, 9/20, \infty \}.$$ Under the action of Galois this splits up into $9$ orbits, namely:
\begin{align*}&\{0, 1/9, 1/7, 1/3\}, \{1/8, 1/4, 3/8, 3/4\}, \\
&\{1/6, 1/2\}, \{1/5, 4/5\}, \{2/5, 3/5\},\\
&\{1/10\}, \{3/10\}, \{3/20\},\{7/20\},\{9/20\},\{\infty\}.\end{align*}

This means that $C_1(20)^{(3)}(\Q)$ consists of $\binom{6+3-1}{3} = 56$ divisors supported at the $6$ rational cusps and $6 \cdot 3$ divisors of the form $P_1 + P_2$ where $P_1$ is one of the $6$ rational cusps and $P_2$ is the sum of the points in either of the three pairs of Galois conjugate cusps. By explicit computation one checks that 12 of the 18 effective divisors of degree 3 not supported at the rational cusps do not occur as the pole of a function of degree $3$ and that the other 6 are linearly equivalent to a divisor supported on the rational cusps, so that up to automorphisms of $\PP^1$ every function of degree $3$ has a pole divisor supported on the rational cusps. It turns out that of the 56 effective divisors of degree 3 supported on the rational cusps, there are exactly 14 that are the pole of a degree $3$ function. These 14 divisors are divided into 4 linear equivalence classes of size 2 and 2 linear equivalence classes of size 3, for a total of 6 linear equivalence classes, hence up to the automorphisms of $\PP^1$ there are exactly 6 functions of degree $3$ on $X_1(20)$. The 4 linear equivalence classes of size $2$ form one orbit under the diamond operators, as do the 2 linear equivalence classes of size $3$. So that the $6$ functions form only $2$ distinct orbits under the action of the diamond operators. Finally the explicit equations for these functions where obtained by finding an $f$ whose divisor is the difference between divisors in the same linear equivalence class. 

As in the proof of \Cref{lem:16_1}, the degree 3 functions $g_i$ will be written down as elements
$$g_i\in\Q(x)[y]/(y^3 - x^3y^2 -  x^2y + x) \simeq \Q(X_1(20))$$
and $f_i \in \Q(x)[t]$ will denote the minimal polynomial of $g_i$ over $\Q(x)$.

The explicit equations for the 2 maps $X_1(20)\rightarrow \PP^1$ of degree $3$ that we obtain are:
\begin{align}
\label{maps:20}
g_1(x,y)&=y - 1,\\
\label{fi:20_1}
f_1(y,t)&=t - y + 1,\\
\label{maps:20_2}
g_2(x,y)&=\frac{xy+y}{y+1},\\
\label{fi:20_2}
f_2(y,t)&=(y^3 - y)t^3 + (-3y^2 + 2y + 1)t^2 + (3y^2 + 2y - 1)t - y^2 - 2y - 1
.
\end{align}
\end{proof}

\begin{remark}
Notice that from the computations in the proof of \Cref{lem:20_1} one can also see that the image of $\mu_3: C_1(20)^{(3)}(\Q) \to \Pic^3 X_1(20)(\Q)$ has 60 elements and hence is surjective.
\end{remark}

\begin{lemma}
\label{lem:rat_16}
The only elliptic curve $E$ with $\Cn {16}$ torsion over a cyclic cubic field $K$ is
$$y^2 + axy +
by = x^3 +bx^2,$$
where
\begin{equation}a=\frac{-11\alpha^2 + 2543\alpha + 2240}{2232}, b=  \frac{481\alpha^2 - 2465\alpha -
    376}{155682},
\label{exc_curve}
\end{equation}
$\alpha$ is a root of $x^3 - 8x^2 - x + 8/9$ and $K=\Q(\alpha)$.
\end{lemma}
\begin{proof}
Since we've proved in \cref{cor:inverse_image} that all cubic points on $X_1(16)$ can be obtained by the action of diamond operators on the inverse images of $\PP^1(\Q)$ under the maps \eqref{mp:16_start}, \eqref{mp:16_2}, \eqref{mp:16_3} and \eqref{mp:16_end}, this allows us to describe all elliptic curves with $\Cn {16}$ torsion over cubic fields as a union of parametric families $E_t/K_t$, where each curve $E_t$ is defined over the cubic field $K_t$.

Thus we have to find all values $t$ such that $\Delta(K_t)$ is a square, since these will correspond to values of $t$ for which $K_t$ is cyclic. This is equivalent to finding all the values $t$ such that the discriminant $\Delta(f_i)$ is a square for $1 \leq i \leq 4$, and where the $f_i$ are given in \eqref{maps:16_1}, \eqref{maps:16_2}, \eqref{maps:16_3} and \eqref{maps:16_4}.

This leads us to consider the following $4$ hyperelliptic curves; each curve $C_i$ is obtained by having $y^2=\Delta(f_i)$.
\begin{align}
C_1 = C_1(16):y^2&=128t^7 - 240t^6 + 112t^5 - 12t^4 + 80t^3 - 88t^2 + 32t - 4,\\
C_2 = C_2(16):y^2&= t^8 - 12t^7 + 54t^6 - 112t^5 + 97t^4 -
    32t^3 + 4t^2 - 4t.\\
C_3 = C_3(16):y^2&=t(t^7 - 2t^5 + 16t^4 - 15t^3 + 36t^2 - 32t + 4),\\
C_4 = C_4(16):y^2&= t^6 - 4t^5 - 4t^4 - 40t^3 + 20t^2 -
    32t,
\end{align}

The curve $C_4$ has genus 2, while the other three curves have genus 3. Using 2-descent for hyperelliptic curves in Magma we get that the Jacobian of the curve $C_1$ has rank 0, while the Jacobians of the remaining 3 curves have rank 1.

Since the Jacobian of the curve $C_1$ has rank 0, it is easy to find all the rational points on $C_1$. We find that $\# \Aut(C_1)=4$, so $C_1$ has automorphisms which are not the identity and not the hyperelliptic involution. Let $w$ of $C_1$ be the automorphism defined by $w(x,y)=\left(\frac{x}{x-1}, \frac{y}{x-1}\right)$. The quotient curve $C_1/\diam{w}$ is isomorphic to the elliptic curve with LMFDB label 26.a2 (Cremona label 26a1). This elliptic curve has 3 rational points and the only rational points on $C_1$ in the preimages of the rational points on 26a2 are
$$C_1(\Q)=\{(1/2, 0), \infty\}.$$
We obtain that the cubic points corresponding to $C_1(\Q)$ are cusps.

Computing the rational points on the curve $C_4$ is straightforward - since $C_4$ has genus 2, the Chabauty method is implemented in Magma, and Magma returns all the rational points. We get that $$C_4(\Q)=\left\{(-1/4,\pm 201/4),(0,0), \pm \infty\right\}.$$
From the value $t=-1/4$ we get the exceptional curve \eqref{exc_curve}, while the values $t=0$ and $\pm \infty$ correspond to cusps of $X_1(16)$.

Finding $C_2(\Q)$ and $C_3(\Q)$ is much more technically difficult and requires the combined use of the Mordell-Weil sieve and the Chabauty method. An additional difficulty is that the Chabauty method is not implemented in Magma for genus $>2$ curves; this required us to write our own implementation for the Chabauty method for genus $>2$ hyperelliptic curves in Magma. As the methods used in the determination of the rational points on $C_2$ and $C_3$ is similar to what we use in determining the points on a curve in \Cref{curves:20}, we give a detailed explanation of the determination of rational points on $C_2$ and $C_3$ in \Cref{sec:computations} and in particular in \Cref{thm:rat_pts}.

We obtain
$$C_2(\Q)=\{(0,0),\pm\infty \} \text{ and } C_3(\Q)=\{(0,0),\pm\infty \},$$
and that the cubic points on $X_1(16)$ corresponding to $C_2(\Q)$ and $C_3(\Q)$ are cusps.
\end{proof}

\begin{lemma}
\label{curves:20}
There are no elliptic curves with torsion $\Cn {20}$ over cyclic cubic fields.
\end{lemma}
\begin{proof}
Since we've proved in Lemma \cref{cor:inverse_image} that all cubic points on $X_1(20)$ can be obtained by the action of diamond operators on the inverse images of $\PP^1(\Q)$ under the maps \eqref{maps:20} and  \eqref{maps:20_2}, this allows us to describe all elliptic curves with $\Cn {20}$ torsion over cubic fields as a union of parametric families $E_t/K_t$, where each curve $E_t$ is defined over the cubic field $K_t$.

Thus, as in the proof of \Cref{lem:rat_16} we have to find all values $t$ such that $\Delta(K_t)$ is a square, since these will correspond to values of $t$ for which $K_t$ is cyclic. This is equivalent to finding all the values $t$ such that the discriminant $\Delta(f_i)$ is a square for $i=1,2$, and where the $f_i$ are given in \eqref{fi:20_1}, and \eqref{fi:20_2}.

We obtain the following hyperelliptic curves; each curve $C_i$ is obtained by taking $y^2=\Delta(f_i)$:
\begin{align}
C_1 = C_1(20) :y^2&=-27t^8 + 22t^4 + 5,\\
C_2 = C_2(20) :y^2&=t^{10} - 6t^9 + 15t^8 - 32t^7 + 51t^6 -
    54t^5 + 65t^4 - 64t^3 + 24t^2 - 4t.
\end{align}
The genus 3 curve $C_1$ has an obvious degree $2$ map $f$ to the rank 0 elliptic curve
$$E:y^2=-27t^4+t^2+5.$$
We find that $E(\Q)\simeq \Cn 6$ and the only $\Q$-rational points in $f^{-1}(E(\Q))$ are
$$C_1(\Q)=\{(\pm 1, 0)\}.$$

Finding $C_2(\Q)$ is much more technically difficult and requires the combined use of the Mordell-Weil sieve and the Chabauty method. We describe these methods and how they were used in detail in \Cref{sec:computations} and in particular in \Cref{thm:rat_pts} (as they are similar to the ones used in the proof of \Cref{lem:rat_16}).

We obtain
$$C_2(\Q)=\{(0,0),\pm\infty \}$$
and that the cubic points on $X_1(20)$ corresponding to $C_2(\Q)$ are cusps.
\end{proof}

\section{Explicit rational point computations}
\label{sec:computations}
In the previous section we determined the rational points on $C_1(16),C_4(16)$ and $C_1(20)$. This was relatively easy since they either had a map to a rank 0 elliptic curve whose rational points are easy to compute, or were of genus 2 and the Jacobian had rank 1 so that Magma could compute the set of rational points using explicit Chabauty. For the remaining three curves $C_2(16)$, $C_3(16)$ and $C_2(20)$ the $r < g$ condition needed for explicit Chabauty is satisfied, so from a theoretical point of view it should be easy to compute all rational points on these three curves using Chabauty-Coleman (see \cite{mccallum-poonen}) in combination with some Mordell-Weil sieving (see \cite{bruin-stoll}). However a general implementation for explicit Chabauty for hyperelliptic curves of genus $>2$ is not implemented in Magma yet so we briefly describe how we manually computed all rational points on these three curves. Before the actual computations, we briefly describe how to do Mordell-Weil sieving without knowing the full Mordell-Weil group.

\subsection{Mordell-Weil sieving with partial Mordell-Weil information}
Mordell-Weil sieving, as explained for example in \cite{bruin-stoll}, is a useful tool that helps find all rational points on a curve $C$ using the Mordell-Weil group of its Jacobian $J$. In order to simplify the exposition\footnote{At the cost of some technical difficulties one could drop the assumption on the existence of a rational point and the assumption that the primes in $S$ are of good reduction. Additionally one can even let $S$ consist of prime powers instead of primes} we fix a rational point $\infty \in C(\Q)$ and a finite set of primes $S = \lbrace p_1,\ldots,p_k \rbrace$ of good reduction for $C$ and use the point $\infty$ as a base point for the maps $\mu : C(\Q) \to J(\Q)$ and $\mu : C(\F_p) \to J(\F_p)$ for all primes $p \in S$. Then for every integer $N$ we have the following commutative diagram:
$$
\xymatrix{ C(\Q) \ar[r] \ar[d]& J(\Q)/NJ(\Q) \ar[d]\\
\prod_{p \in S} C(\F_p)\ar[r] &  \prod_{p \in S} J(\F_p)/NJ(\F_p)}.
$$
The usefulness of this commutative diagram is that one can show that for a fixed $q\in S$ one has that $C(\Q) \to C(\F_q)$ factors via \begin{align}\prod_{p \in S} C(\F_p) \times_ {\prod_{p\in S} J(\F_p)/NJ(\F_p)} J(\Q)/NJ(\Q) \to C(\F_q). \label{eq:mws}\end{align}
In particular one can show that certain points in $C(\F_q)$ have no rational points reducing to them if the map in \cref{eq:mws} is not surjective. However in order to compute the fiber product in \cref{eq:mws} one has to know all of $J(\Q)$. In practice (for example for the curves relevant to his article) one sometimes can find explicit generators of a finite index subgroup $ \Gamma \subseteq J(\Q)$. Even though computing $J(\Q)$ using the knowledge of $\Gamma$ is a finite computation, this is not always that easy in practice. This is where the following lemma comes in useful:
\begin{lemma}\label{lem:mws}
	Let $C$ be a curve with Jacobian $J$, $S := \lbrace p_1,\ldots,p_k \rbrace$ a finite set of primes of good reduction, $N$ a positive integer, $\Gamma \subseteq J(\Q)$ a subgroup of finite index, $e$ such that $eJ(\Q) \subseteq \Gamma$ and $d := gcd(N,e)$, then for every $q \in S$ the map $C(\Q) \to C(\F_q)$ factors via:
	\begin{align}{\rm MWS}_{\Gamma,N,S,d} := \prod_{p\in S} C(\F_p) \times_ {\prod_{p\in S} J(\F_p)/NJ(\F_p)} \Gamma/N\Gamma \to C(\F_q). \label{eq:mws2}\end{align}
	where in the above fiber product the map $\prod_{p\in S} C(\F_p) \to {\prod_{p\in S} J(\F_p)/NJ(\F_p)}$ is $[d] \circ \mu$ on each coordinate and $\Gamma/N\Gamma \to {\prod_S J(\F_p)/NJ(\F_p)}$ is just the reduction map.
\end{lemma}
\begin{proof}
	The definition of $d$ assures the existence of an integer $f$ such that $d \equiv ef \mod N$, so the lemma follows from the following commutative diagram:
	$$\xymatrix{C(\Q) \ar[r]^{\mu}  \ar[d]& J(\Q)/NJ(\Q) \ar[r]^{[ef]=[d]} \ar[d] & \Gamma/N\Gamma \ar[d]\\
	\prod_{p\in S} C(\F_p) \ar[r]^{\mu} & \prod_{p\in S} J(\F_p)/NJ(\F_p) \ar[r]^{[ef]=[d]} & \prod_{p\in S} J(\F_p)/NJ(\F_p)
	}.$$
\end{proof}
The above lemma allows one to circumvent the computation of $J(\Q)$ because one can often compute $d$ without knowing $J(\Q)$ or $e$.
\begin{example}\label{ex:mws}
Suppose $J(\Q)$ has rank 1 and $x \in J(\Q)$ is a generator of the free part. Let $d_{tors}$ be such that $d_{tors} J(\Q)_{tors} \subseteq \Gamma_{tors}$, and let $d_{free}$ be such that $N/d_{free}$ is the order of $x$ in  $\prod_{p\in S} J(\F_p)/NJ(\F_p)$ then one can find an $e$ such that $d = gcd(lcm(d_{tors}, N),d_{free})$.
\end{example}
\subsection{Rational points on $C_2(16), C_3(16)$ and $C_2(20)$}
\begin{theorem}\label{thm:rat_pts}
The hyperelliptic curves
\begin{align}
C_2(16):y^2&= t^8 - 12t^7 + 54t^6 - 112t^5 + 97t^4 -
32t^3 + 4t^2 - 4t,\\
C_3(16):y^2&=t(t^7 - 2t^5 + 16t^4 - 15t^3 + 36t^2 - 32t + 4) \textrm{ and }\\
C_2(20) :y^2&=t^{10} - 6t^9 + 15t^8 - 32t^7 + 51t^6 -
54t^5 + 65t^4 - 64t^3 + 24t^2 - 4t
\end{align}
each have exactly $3$ rational points, namely the affine point $P_1 = (0,0)$ and the two points $P_2 = (1:1:0), P_3 = (1:-1:0) \in C(\Q)$ at infinity. Furthermore, their Jacobians are of rank 1 over $\Q$ and the three rational points generate a finite index subgroup of the Jacobian.
\end{theorem}
\begin{proof}
	In the parts of the proof below which are the same for all three curves we will just write $C$ for the curve and $J$ for the Jacobian of $C$.

	Magma has an implementation of two-descent for hyperelliptic curves, and the result of a two-descent computation is that the rank is at most one for each of the three curves, and one easily checks that $[P_1 - P_2] = -[P_1-P_3]$ is a point of infinite order on $J(\Q)$ for all three curves. So this proves the "furthermore" part of \Cref{thm:rat_pts}. What remains is to do explicit Chabauty and then some Mordell-Weil sieving to actually compute the rational points.
	
	Since these curves are hyperelliptic, it is easy to give a basis for the K\"ahler differentials, namely $$\omega_0 := \frac {x^0{\rm d}x} {y}, \omega_1 := \frac {x^1{\rm d}x} {y}, \ldots, \omega_{g-1} := \frac {x^{g-1}{\rm d}x} {y},$$ where $g$ is the genus of the curve. And if $p$ is a prime of good reduction this is even a $\Z_p$-basis.
	
	The next step in the Chabauty method is to find $p$-adic differentials $\omega \in H^0(C,\Omega^1_{C/\Z_p})$ that vanish on $J(\Q)$ under the $p$-adic integration pairing
	\begin{align*}
	\left \langle \_,\_ \right \rangle :  J(\Q_p) \times H^0(C,\Omega^1_{C/\Z_p}) &\to \Q_p,\\
	(D,\omega)	 & \mapsto \int_0^D \omega.
	\end{align*}
	Since $[P_1-P_2]$ generates a subgroup of $J(\Q)$ of finite index one has that a $p$-adic 1-form $\omega$ vanishes on all of $J(\Q)$ if and only if $\int_{P_2}^{P_1} \omega = 0$. Explicitly computing $\int_{P_2}^{P_1} \omega_i$ up to some $p$-adic precision allows one to find $a_i \in \Z_p$ such that
	$\sum_{i=0}^{g-1} a_i \int_{P_2}^{P_1}  \omega_i = \int_{P_2}^{P_1} \sum_{i=0}^{g-1} a_i  \omega_i =0$. For $C=C_2(16)$ we use $p=5$ from now on, and for $C=C_3(16)$ and $C=C_2(20)$ we use $p=3$. The result of this computation gives
	\begin{align*}
	\int_{P_2}^{P_1} \omega \equiv 0 \mod p^5 = 5^5 & & \text{for } \omega &:= \omega_2 + 2983 \omega_0 \in \Omega^1_{C_2(16)/\Z_5} \\
	\int_{P_2}^{P_1} \omega \equiv 0 \mod p^5 = 3^5 & & \text{for } \omega &:=\omega_2 + 118 \omega_0  \in \Omega^1_{C_3(16)/\Z_3} \\
	\int_{P_2}^{P_1} \omega \equiv 0 \mod p^5 = 3^5 & & \text{for } \omega &:= \omega_3 + 4 \omega_1 + \omega_0\in \Omega^1_{C_2(20)/\Z_3} \\
	\end{align*}
	One easily verifies that these $\omega$ have no zeros at the $P_{i,\F_p} \in C(\F_p)$ so that the conclusion of the Chabauty computation is that the residue discs of the $P_{i,\F_p}$ contain exactly one rational point. What remains to do in order to show that $C(\Q)=\lbrace P_1,P_2,P_3\rbrace$ is to show that there are no rational points reducing to the points in $C(\F_p) \setminus \lbrace P_{1,\F_p},P_{2,\F_p},P_{3,\F_p}\rbrace$; this is done using Mordell-Weil sieving.
	
	Computing $\#J(\F_p)$ for the first 20 primes of good reduction shows that $\#J(\Q)_{tors}$ is either 1 or 3. Instead of computing the exact 3-torsion we will use \Cref{lem:mws} in order to do Mordell-Weil sieving without computing the 3-torsion. In all three cases the finite index subgroup $\Gamma$ will be the group
	$$\Gamma := \diam{[P_1-P_2]} \subseteq J(\Q).$$

	{\bf Mordell weil sieving for $C_2(16)$.} The primes we will use to sieve with are $5$ and 11. The group structures of the Jacobian at these primes are $J(\F_5) \cong \Z/3\cdot 11^2 \Z$ and $J(\F_{11}) \cong \Z/2\Z \times  \Z/2^5 \cdot 3\cdot 11 \Z$. For $N$ we will take $22$ in order to exploit the interference between the $11$ torsion of $J(\F_5)$ and $J(\F_{11})$ and the non-cyclicity of the $2$ torsion of $J(\F_{11})$.
	
	The image of $x=[P_1 - P_2]$ in $J(\F_5)/22J(\F_5)$ is of order 11 and a generator. The 11 points in $C_2(16)(\F_5)$ map to $0x, 2x, 2x, 4x, 4x, 5x, 5x, 7x, 7x, 9x$ and $10x$ in $J(\F_5)/22J(\F_5)$  respectively, where $0x,9x$ and $10x$ correspond to the known 3 rational points. The image of $x$ in $J(\F_{11})/22J(\F_{11})$ is of order 22 so that as in \Cref{ex:mws} we can use $d=1$. It is also clear that $x$ generates a subgroup of index 2. Of the 16 points in $C_2(16)(\F_{11})$ there are 8 that map to a multiple of $x$, the multiples are $0, 3, 10, 10, 17, 20, 21$ and $21$. In particular the image of  \begin{align}{\rm MWS}_{\Gamma,22,\set{5,11},1}  \to C(\F_5). \end{align} consists exactly of the rational points so that we are done by \Cref{lem:mws}.
	
	\vspace{0.5em}
	{\bf Mordell weil sieving for $C_3(16)$.} In this case only a single prime, namely $p=3$, is used for sieving, but we use the auxiliary prime $17$ in order to show that $\Gamma \subseteq J(\Q)$ is saturated at $2$. The group structures of the Jacobian at these two primes are $J(\F_3) \cong \Z/2\cdot 3^3 \Z$ and $J(\F_{17}) \cong  \Z/2 \cdot 3^2\cdot 293 \Z$.

	In this case there is the nice coincidence that the 3 points in $C_3(16)(\F_3)$ having a known rational point in their residue class are all sent to 0 in $J(\F_3)/2J(\F_3) \cong \Z/2\Z$ while the two other points in $C_3(16)(\F_3)$ are sent to the nonzero element. Furthermore, the image of $\Gamma$ in $J(\F_3)/2J(\F_3)$ is 0 so that just using the Mordell-Weil sieve at the prime $3$ using $N=2$ gives the desired answer as soon we can prove that the finite index submodule $\Gamma\subseteq J(\Q)$ is saturated at 2. But the latter is indeed saturated at 2 since $J(\Q)$ has no two torsion and the generator of $\Gamma$ is sent to the nonzero element in $J(\F_{17})/2J(\F_{17})$.
	
	\vspace{0.5em}
	{\bf Mordell weil sieving for $C_2(20)$.} The primes we will use to sieve with are $3$ and 37. The group structures of the Jacobian at these primes are $J(\F_3) \cong \Z/3\cdot 47 \Z$ and $J(\F_{37}) \cong \Z/3\Z \times  \Z/3 \cdot 47 \cdot 4651 \Z$. For $N$ we will take $47$ in order to exploit the interference between the $47$ torsion of $J(\F_3)$ and $J(\F_{37})$.
	
	The image of $x=[P_1 - P_2]$ in $J(\F_3)/47J(\F_3)$ is of order 47 and a hence a generator, so that due to \Cref{ex:mws} we can use $d=1$. The 5 points in $C_2(20)(\F_3)$ map to $0x, 19x, 26x, 45x$ and $46x$ in $J(\F_3)/47J(\F_3)$  respectively, where $0x,45x$ and $46x$ correspond to the known rational points. Each of the 39 points in $C_2(20)(\F_{37})$ is mapped to a multiple $nx$ of $x$, where the values of $n$ are as follows:
	\begin{align*}&0, 1, 1, 1, 2, 2, 3, 3, 6, 6, 7, 9, 12, 14, 14, 16, 17, 18, 22, 23, 27, \\
	&28, 29, 31, 31, 33, 36, 38, 39, 39, 42, 42, 43, 43, 44, 44, 44, 45, 46.\end{align*}
	Now one sees that $19$ and $26$ are not in this list, so that in particular the image of	\begin{align}{\rm MWS}_{\Gamma,47,\set{3,37},1}  \to C(\F_3). \label{eq:mws3}\end{align} consists exactly of the rational points so we are done by \Cref{lem:mws}.
\end{proof}
\section{Torsion groups over complex and over totally real and non-Galois cubic fields}

\begin{theorem}\label{mainthm2}
Suppose $E/K$ be an elliptic curve over a complex cubic field $K$. Then $E(K)_{tors}$ is isomorphic to one of the following groups:
$$\Cn n \text{ for }n= 1\ldots 16, 18, 20$$
$$\Cn 2 \times \Cn {2m} \text{ for }m=1\ldots 6.$$
Each of the listed groups occurs for infinitely many distinct $j$-invariants.
\end{theorem}
\begin{proof}
We prove this theorem by explicit computation. We use the infinite families $E_t/K_t$ from \cite{jky}, where for each torsion group from the statement of the theorem, an infinite family of elliptic curves with such torsion is given. For each family we have to determine whether there exists a value $t\in \Q$ such that $\Delta(K_t)<0$ (which is equivalent to $K_t$ being complex).

If there exists at least one value $t\in \R$ such that $\Delta(K_t)<0$ then this immediately means that there are infinitely many $t\in \Q$ such that $\Delta(K_t)<0$. For all the families we explicitly find a value $t$ such that $\Delta(K_t)<0$ and we are done.

\end{proof}

\begin{theorem}\label{mainthm3}
Suppose $E/K$ be an elliptic curve over a totally real cubic field $K$ whose Galois group is $S_3.$ Then $E(K)_{tors}$ is isomorphic to one of the following groups:
$$\Cn n \text{ for }n= 1\ldots 16, 18, 20$$
$$\Cn 2 \times \Cn {2m} \text{ for }m=1\ldots 6.$$
Each of the listed groups occurs for infinitely many distinct $j$-invariants.
\end{theorem}
\begin{proof}
Recall that totally real cubic fields $K$ that are not Galois over $\Q$ have $\Delta (K)>0$ and $\Delta(K)$ is not a square. The proof of this theorem follows from similar computations as \Cref{mainthm2}. We use the same families $E_t/K_t$ as in \Cref{mainthm2}, and only need to find values $t\in \R$ such that $\Delta(K_t)>0$, by using the same argument as in the proof of \Cref{mainthm2}. Once we find such a value, by the same argument as before, we see that there are infinitely many $t\in \Q$ such that $\Delta(K_t)>0$.

It remains to prove that for infinitely many of those $t$ such that $\Delta(K_t)>0$, it is true that $\Delta(K_t)$ is not a square (in $\Q$). Let $$C_1:y=\Delta(K_t)\text { and } C_2:y^2=\Delta(K_t).$$
Now there is an obvious degree 2 map $\phi:C_2\rightarrow C_1$ and hence $\phi(C_2(\Q))$ is a \textit{thin set} (see \cite[Chapter 9]{ser}) in $C_1(\Q)$ and $C_1(\Q) \backslash \phi(C_2(\Q))$ is dense, hence there are infinitely many $t\in \Q$ such that $\Delta(K_t)>0$.

We manage to find a $t\in \Q$ such that $\Delta(K_t)>0$ for every group apart from $\Cn 2\times \Cn {12}$. For $\Cn 2\times \Cn {12}$, we find a totally real cubic field $K$ whose Galois group is $S_3$ over which $X_1(2,12)$ (which is and elliptic curve) has positive rank.

\end{proof}

\end{document}